\documentclass[11pt]{amsart}
\usepackage{pxfonts}
\usepackage[bookmarks=false]{hyperref}
\usepackage{amsmath}
\newtheorem{theorem}{Theorem}[section]
\newtheorem{proposition}{Proposition}[section]
\newtheorem{lemma}{Lemma}[section]

\theoremstyle{definition}
\newtheorem*{definition}{Definition}{}

\theoremstyle{remark}

\newcommand{\spbp}{\mathcal{\sqrt{-1}\partial \bar{\partial}}}

\numberwithin{equation}{section}
\usepackage{hyperref}
\hypersetup{colorlinks,citecolor=blue,plainpages=false,hypertexnames=false}
\begin{document}
\title[Coupled K\"ahler-Einstein metrics]{Existence of coupled K\"ahler-Einstein metrics using the continuity method}
\author{Vamsi Pritham Pingali}
\address{Department of Mathematics, Indian Institute of Science, Bangalore, India - 560012}
\email{vamsipingali@math.iisc.ernet.in}
\begin{abstract}
In this paper we prove the existence of coupled K\"ahler-Einstein metrics on complex manifolds whose canonical bundle is ample. These metrics were introduced and their existence in the said case was proven by Hultgren and Nystr\"om using calculus of variations. We prove the result using the method of continuity. In the process of proving estimates, akin to the usual K\"ahler-Einstein metrics, we reduce existence in the Fano case to a $C^0$ estimate.
\end{abstract}
\maketitle
\section{Introduction}
\indent Let $(X,\omega_0)$ be a compact K\"ahler manifold which is either Fano ($c_1(X)>0$) or anti-Fano ($c_1(X)<0$). Consider the following equations (the ``coupled K\"ahler-Einstein equations") on $X$, originally introduced in \cite{David}.
\begin{gather}
\mathrm{Ric}(\omega_1) = \mathrm{Ric}(\omega_2) = \ldots = \pm \displaystyle \sum \omega_i, \label{coupledsys}
\end{gather}
where $\omega_i$ are K\"ahler metrics to be solved for in given K\"ahler classes $[\theta_i]$ satisfying $\displaystyle \pm \sum_i  [\theta_i] = c_1 (X)$. These equations seem vaguely reminiscent of the bimetric theories of gravity (see \cite{bim} and the references therein). \\
\indent It can easily be shown that \ref{coupledsys} is equivalent to the following system of Monge-Amp\`ere PDE if $\omega_0$ satisfies $\mathrm{Ric}(\omega_0)=\pm \displaystyle \sum _i \theta_i$. (This can be arranged using Yau's solution of the Calabi conjecture \cite{Yau}.)
\begin{gather}
(\theta_i+\spbp \phi_i) ^n = C_i e^{\displaystyle \mp \sum _i \phi_i} \omega_0^n 
\label{coupledMA}
\end{gather} 
for smooth functions $\phi_i$ satisfying $\sup \phi_2 = \sup \phi_3 =\ldots = \sup \phi_n = 0$ where $C_i = \frac{\int \theta_i ^n}{\int \omega_0 ^n}$. In \cite{David} the following existence result was proven for anti-Fano $X$.
\begin{theorem}[Hultgren-Nystr\"om]
Let $(X,\omega_0)$ be a compact K\"ahler manifold which is anti-Fano. Let $[\theta_i]$ be K\"ahler classes such that $\displaystyle  \sum _i [\theta_i] =-c_1(X)$. Then there exist unique K\"ahler metrics $\omega_i \in [\theta_i]$ such that \begin{gather}
\mathrm{Ric}(\omega_1) = \mathrm{Ric}(\omega_2) = \ldots = -\displaystyle \sum _i \omega_i
\label{KEeqn}
\end{gather}
\label{mainthm}
\end{theorem}
\indent Hultegren and Nystr\"om  proved theorem \ref{mainthm} using calculus of variations. In this paper we prove this theorem using the method of continuity. To do this we establish the following \emph{a priori} estimates.
\begin{theorem}
Let $(X,\omega_0)$ be a compact K\"ahler manifold that is either Fano or anti-Fano such that $\omega_0$ satisfies $\mathrm{Ric}(\omega_0)=\pm \displaystyle \sum _i \theta_i$ where $\theta_i$ are K\"ahler forms such that $\displaystyle \pm \sum _i [\theta_i] =c_1(X)$. Let $\phi_i$ be a smooth solution of the following system of equations.
\begin{gather}
(\theta_i+\spbp \phi_i) ^n = C_i e^{\displaystyle \mp \sum _i t_i \phi_i} \omega_0^n 
\label{paracoupledMA}
\end{gather} 
where $C_i = \frac{\int \theta_i ^n}{\int \omega_0 ^n}$ and $0\leq t_i \leq 1$. 
\begin{enumerate}
\item If $X$ is anti-Fano then $\Vert \phi_i \Vert _{C^{2,\alpha}} \leq C$ where $C$ is bounded uniformly.
\item If $X$ is Fano then $\Vert \phi_i \Vert _{C^{2,\alpha}} \leq C$ where $C$ depends on $\Vert \phi_1 \Vert _{C^0}$. 
\end{enumerate} 
\label{estimatesthm}
\end{theorem}
Note that at $t_i=0\ \ \forall \ i$, the functions $\phi_i=0$ solve the equations. By theorem \ref{estimatesthm} the set of $t_i$ for which there exists a solution is closed for anti-Fano manifolds. Theorem \ref{mainthm} follows from the following openness result.
\begin{theorem}
The set of $0\leq t < 1$ for which there exists a unique smooth solution to the following system is open.
\begin{gather}
\mathrm{Ric}(\theta_{1\phi_1}) = \mathrm{Ric}(\theta_{2\phi_2}) = \ldots = \pm \left ( \displaystyle \sum t \theta_{i \phi_i} +  \sum (1-t) \theta_i \right )
\label{mainsyscont}
\end{gather}
\label{mainthmopen}
\end{theorem}
Notice that theorems \ref{estimatesthm} and \ref{mainthmopen} reduce the problem for Fano manifolds to the $C^0$ estimate just as in the usual K\"ahler-Einstein case. In \cite{David} an obstruction to solving the equation akin to K-stability was discovered for Fano manifolds. It is interesting to see if the corresponding $C^0$ estimate can be proven along this continuity path using techniques of \cite{CDS1,CDS2,CDS3,Datar}.\\

\emph{Acknowledgements} : The author thanks David Witt Nystr\"om for answering some questions. The author is especially grateful to Jakob Hultgren for pointing out a crucial error and proof-reading the solution to the same.
\section{\emph{A priori} estimates on solutions to equation \ref{coupledMA}}
\indent As is often the case in fully nonlinear PDE, we prove lower order estimates and improve upon them. In what follows unless clarity demands otherwise, we denote arbitrary uniform (in the time parameters in the method of continuity) constants by $C$.\\
\indent We first prove a $C^0$ estimate in the anti-Fano case. 
\begin{lemma}
If $c_1(X)<0$ then any smooth solution $\phi_i$ satisfying $\sup \phi_2 = \sup \phi_3 =\ldots = 0$ of the system 
\begin{gather}
(\theta_i+\spbp \phi_i) ^n = C_i e^{\displaystyle \sum _i t_i \phi_i} \omega_0^n, 
\label{antifanoeqn}
\end{gather}
where $C_i = \frac{\int \theta_i ^n}{\int \omega_0 ^n}$ and $0\leq t_i \leq 1$ satisfies $\Vert \phi _i \Vert_{C^0} \leq C$.
\end{lemma}
\begin{proof}
 If $\vert \phi_1 \vert_{C^0} \leq C$ then by the assumption that $\phi_i \leq 0 \ \forall \ i \geq 2$, and either the Alexandrov-Bakelmann-Pucci (ABP) maximum principle \cite{Blocki} or $L^p$ stability for $p>1$ \cite{Kolo} we can see that $\Vert \phi_i \Vert_{C^0} \leq C$ for all $1\leq i\leq n$. In addition, the maximum principle shows that $\phi_1 \geq -C$. So we just need to prove that $\phi _1 \leq C$. \\
\indent Choosing a positive Green's function $G$ for the Laplacian of $\omega_0$, we see using the representation formula (page 49 in \cite{Tian} for instance) that
\begin{gather}
u(x)-C \leq \frac{\int u \omega_0 ^n}{V},
\end{gather}
for every $u$ satisfying $\spbp u \geq -C \omega_0$, where $V$ is the volume of $\omega_0$. Taking $u=\displaystyle \sum_ i t_i \phi_i$ and using Jensen's inequality we get
\begin{gather}
\displaystyle \sum _i t_i \phi_i (x) - C \leq \ln \left ( \int e^{\sum t_i \phi_i} \omega_0 ^n \right ) \nonumber \\
\Rightarrow \displaystyle \sum _i t_i \phi_i (x) \leq C. 
\end{gather}	
Therefore, $\Vert e^{\sum t_i \phi_i} \Vert _{L^p} \leq C_p$ for all $p>1$. Thus by the the ABP estimate as before we see that $-C \leq \phi_i \leq C$. 
\end{proof}
\indent We proceed to prove a bound on the Laplacian in both, the Fano, and the anti-Fano cases.
\begin{lemma}
Any smooth solution $\phi_i$ of the system \ref{coupledMA} satisfies $\Vert \Delta \phi _i \Vert \leq C$.
\end{lemma}
\begin{proof}
Let $u_i  = e^{-\lambda \phi_i} (n+\Delta_{\theta_i} \phi_i)$. We shall assume that $\Vert \phi_i \Vert_{C^0} \leq  C$ in what follows. Just as in Yau's proof \cite{Yau} we write the following inequality (inequality 2.3 from \cite{Chen} for instance) for solutions of $\omega_v ^n= (\omega+\spbp v) ^n =e^{F-\lambda v} \omega^n$
\begin{gather}
\Delta_{\omega_v} (\exp(-C_1 v)(n+\Delta v)) \geq \exp(-C_1 v) \left [ \Delta F - C_2- C_1(n+\Delta v) \right ] \nonumber \\+ \exp \left (-C_3 v -\frac{F}{n-1} \right )C_4 (n+\Delta v)^{n/(n-1)}.
\end{gather} 
Replacing $C_1$ by $\lambda$, $F$ by $F+a\displaystyle \sum _{i \neq j}t_j \phi_j$, $\omega$ by $\theta_i$, and $v$ by $\phi_i$ in the above inequality we get (after a couple of easy estimates) the following. Note that $a = 1$ or $a=-1$ depending on whether the manifold is anti-Fano or Fano respectively.
\begin{gather}
\Delta_{\theta_{\phi_i}} u_i\geq -C+ \tilde{C} u_i ^{n/(n-1)}+ \displaystyle a\sum_{i\neq j}t_j e^{-\lambda \phi_i}\Delta _{\theta_i} \phi_ j \nonumber \\
= -C+ \tilde{C} u_i ^{n/(n-1)}+ \displaystyle a\sum_{i\neq j} e^{-\lambda \phi_i} t_j \frac{\theta_i ^{n-1} \wedge \spbp \phi_ j}{\theta_i ^n}. 
\label{inter}
\end{gather}
At this point we analyse the two cases $a= \pm 1$ separately.
\begin{enumerate}
\item $a=1$. In this case we may continue inequality \ref{inter} further as follows.
\begin{gather}
\Delta_{\theta_{\phi_i}} u_i \geq -C+ \tilde{C} u_i ^{n/(n-1)}- \displaystyle \sum_{i\neq j} t_j e^{-\lambda \phi} \frac{\theta_i ^{n-1}  \theta_ j}{\theta_i ^n}. \nonumber
\end{gather}
Therefore by the maximum principle $u_i \leq C \ \forall \ 1\leq i \leq n$.
\item In this case we have the following consequence of inequality \ref{inter}.
\begin{gather}
\Delta_{\theta_{\phi_i}} u_i\geq -C+ \tilde{C} u_i ^{n/(n-1)}-C \displaystyle \sum_{i\neq j} e^{-\lambda \phi_i} t_j \frac{\theta_j ^{n-1} \wedge \spbp \phi_ j}{\theta_j ^n} \nonumber \\
\geq -C + \tilde{C} u_i ^{n/(n-1)} -C \displaystyle \sum_{i\neq j} t_j u_j. 
\label{intertwo}
\end{gather}
Let $\displaystyle \max _X u_i = M_i$. By the maximum principle and inequality \ref{intertwo} we see that for every $i$ we have the following inequality.
\begin{gather}
C\left(1+\sum t_j M_j \right) \geq M_i ^{n/(n-1)}. 
\label{interthree} 
\end{gather}
Summing \ref{interthree} over all $i$ and using Young's inequality $a \leq \epsilon a^{n/(n-1)} + C(n,\epsilon)$ we get (upon choosing a small enough $\epsilon$),
\begin{gather}
C\left(1+\epsilon \sum  M_j^{n/(n-1)} + \sum C(n,\epsilon)  \right) \geq \sum M_i ^{n/(n-1)} \nonumber \\
\Rightarrow M_i \leq C \ \forall \ 1 \leq i \leq n.
\end{gather}
\end{enumerate}
 
\end{proof}
\indent Finally, we need a $C^{2,\alpha}$ estimate in order to complete the proof of theorem \ref{estimatesthm}. Indeed, theorem 1.1 of \cite{Wang} implies the desired $C^{2,\alpha}$ estimate provided $\Vert \phi_i \Vert_{C^1} \leq C$. The latter inequality is true because of the Laplacian bound and $W^{2,p}$ elliptic regularity. We also note that standard elliptic theory (Schauder estimates) and bootstrapping imply that $\Vert \phi _i \Vert_{C^{k,\alpha}} \leq C$ for any $k$.

\section{Uniqueness in the anti-Fano case and openness along the continuity path}
\indent The uniqueness part of theorem \ref{mainthm} was proven in \cite{David} but we prove it again for the convenience of the reader. 
\begin{proposition}
Let $X$ be an anti-Fano manifold. If a solution to the coupled K\"ahler-Einstein equations exists, then it is unique.
\end{proposition}
\begin{proof}
 If $\phi_{i} ^{'} = \phi_i +u ^{(i)}$ is another solution of \ref{coupledMA} then upon subtraction we get
\begin{gather}
L_j^{ab} u^{(j)}_{a\bar{b}} = e^{\sum \phi_i}(e^{\sum u ^{(i)}}-1),
\label{sub}
\end{gather}
where $L_j^{ab} u^{(j)}_{a\bar{b}} = \frac{(\theta_j + \spbp \phi_j + \spbp u^{(j)})^n-(\theta_j+\spbp \phi_j) ^n}{\omega_0 ^n}$. Note that $L_j$ is a positive-definite matrix. Multiplying \ref{sub} by $u^{(j)}$, integrating-by-parts, and summing over $j$ we see that
\begin{gather}
\displaystyle (\sum_j u ^{ (j)} )e^{\sum \phi_i}(e^{\sum u ^{(i)}}-1) \leq 0. 
\end{gather}
This means that $\sum u^{(j)} = 0$ and $\partial u^{(j)} = \bar{\partial} u^{(j)} = 0$. Therefore $u^{(j)} =0 \ \forall \ j$. 
\end{proof}
\indent Now we proceed to prove openness, i.e., theorem \ref{mainthmopen}.\\

\emph{Proof of theorem \ref{mainthmopen}} : Suppose we know that $\omega_i \in [\theta_i]$ solve the system \ref{mainsyscont} for $t$. Then we need to prove that for $t + \delta$ where $\delta$ is in a small open interval, the system can still be solved. We shall in fact consider $t_i$ to be potentially different for different $i$ until the very end of this proof. This is because for the anti-Fano case, one can prove a slightly more general result than the one stated in theorem \ref{mainthmopen}. To this end define the following Banach manifolds.
\begin{definition}
Let $\mathcal{B}^i_1$ be the open subset of $C^{4,\alpha}$ functions $\psi_i$ satisfying $$\displaystyle \int _X \psi_i \omega_i ^n = 0$$ and $$\omega_i + \spbp \psi_i >0.$$ \\
Let $\mathcal{B} _2$ be the subspace of $C^{0,\alpha}$ real $(1,1)$-forms $\eta$ of the form $$\eta = \spbp f$$ where $f$ is a $C^{2,\alpha}$ function satisfying $$\displaystyle \int _X f \omega_0 ^n = 0.$$
\end{definition}
Notice that we have the map $T: U=\Pi _{i=1}^k (\mathcal{B}^i _1 \times [0,1]) \rightarrow V= \mathcal{B} _2 ^k$ given by $$T(\psi_1,t_1,\psi_2,t_2,\ldots) = \left(\mathrm{Ric}(\omega_{1\psi_1}) + a \left (\sum t_i \omega_{i\psi_i}+\sum(1-t_i)\theta_i \right), \ldots \right),$$ where $a=\pm 1$ depending on the sign of $-c_1(X)$. Suppose we take a point $p=(0,t_1,0,t_2,\ldots)$ such that $T(p)=0$. The implicit function theorem states that if $DT_p (v_1,0,v_2,0,\ldots)$ is an isomorphism from $TU$ to $TV$, then $\psi_i$ can be locally solved for in terms of $t_j$ and therefore the set of $t_j$ for which $T=0$ is open. The derivative $DT_p$ is 
\begin{gather}
DT_p (v_1,0,v_2,0,\ldots) = (-\spbp \Delta_{\omega_i}v_1 +a \sum t_i \spbp v_i, \ldots).  
\label{deriv}
\end{gather}
For it to be surjective we need to solve
\begin{gather}
 (-\spbp \Delta_{\omega_i}v_1 +a \sum t_i \spbp v_i, \ldots) = (\spbp f_1, \spbp f_2, \ldots) \nonumber \\
\Rightarrow L(v_1, v_2,\ldots) =  (-\Delta   _{\omega_i}v_1 +a \sum t_i v_i, \ldots) =(f_1, f_2, \ldots).
\label{surj}
\end{gather}
By the Fredholm alternative we simply need to prove that the kernel of $L$ is trivial.  The kernel consists of functions such that
\begin{gather}
\Delta   _{\omega_1}v_1 = a \sum t_i v_i \nonumber \\
\Delta   _{\omega_2}v_2 = a \sum t_i v_i \nonumber \\
\vdots 
\label{ker}
\end{gather}
Note that at $t_i=0 \  \forall \ i $ we see that the kernel is obviously trivial and thus openness holds for small $t_i$. Therefore we may assume without loss of generality that $t_i >0 \ \forall \ i$.  We observe that the normalised volume forms $\frac{\omega_i ^n}{\int \omega_i ^n}$ are all equal (to some form $dvol$) because the Ricci curvatures of $\omega_i$ are equal. Multiplying the $j^{th}$ equation of \ref{ker} by $t_i v_i dvol$ and integrating the left-hand side by parts
\begin{gather}
-\displaystyle \int_X  t_i  \langle \nabla_j v_j , \nabla _j v_i \rangle_j dvol = a \int _X   t_i v_i \sum _k t_k v_k dvol.
\label{usef}
\end{gather}
Taking $i=j$ and summing over all $j$ we get
\begin{gather}
-\displaystyle \int_X  \sum_j t_j \vert \nabla_j v_j \vert_j ^2 dvol = a \int _X    \left(\sum _k t_k v_k \right )^2 dvol.
\label{usefu}
\end{gather}
There are two cases to consider. 
\begin{enumerate}
\item \emph{$X$ is anti-Fano, i.e. $a=1$} : 
Equation \ref{usefu} implies that $\sum t_i v_i = 0$. Therefore by equations \ref{ker} all the $v_i$ are constants and in fact equal to $0$ (because $\int v_i \theta_i ^n = 0$).
\item \emph{$X$ is Fano, i.e., $a=-1$} : A Weitzenb\"ock identity (see page 65 of \cite{Tian} for instance) that
\begin{gather}
\displaystyle \int _X (\Delta _{\omega_i} v_i)^2 dvol \geq \int _X \mathrm{Ric}(\omega_i) (\partial v_i, \bar{\partial} v_i) dvol \nonumber \\
\geq  \displaystyle  \int _X \sum_j t_j  \vert \nabla_i v_i \vert_j ^2 dvol.
\label{prec}
\end{gather}
Assume without loss of generality that none of the $v_i$ are constant. Indeed, if let's say $v_1$ is a constant, then $\Delta_1 v_1=\Delta_i v_i=0$ which by the maximum principle means that all the $v_i$ are constant and in fact $0$ by normalisation. Note that \ref{usef} implies that
\begin{gather}
\displaystyle \int _X \langle \nabla_j v_j , \nabla_j v_i \rangle _j dvol = \int_X \vert \nabla_i v_i \vert_i^2 dvol. 
\label{useftwo}
\end{gather}
Choose normal coordinates for $\omega_i$ at a point $p$.  Further, assume that $\omega_j$ is diagonal at $p$ with eigenvalues $\lambda_{\mu}$. Writing the integrand of the left hand side of \ref{useftwo} at $p$ in the said coordinates we get the following.
\begin{gather}
\langle \nabla_j v_j , \nabla_j v_i \rangle _j (p)= \displaystyle \sum_{\mu} \frac{\partial_{\mu}v_j\bar{\partial}_{\mu}v_i}{\lambda_{\mu}} \nonumber \\
\leq \displaystyle \sqrt{\sum \frac{\vert\partial_{\mu}v_j \vert^2}{\lambda_{\mu}^2} }\sqrt{\sum \vert \partial_{\mu} v_i \vert^2}\nonumber\\
=\vert \nabla_j v_j\vert_i \vert \nabla_i v_i \vert_i.
\label{usefthree}
\end{gather}
Thus using \ref{useftwo}, \ref{usefthree}, and the Cauchy-Schwarz inequality we get
\begin{gather}
\Vert \nabla_j v_j\Vert_i \Vert \nabla_i v_i\Vert_i \geq  \Vert \nabla_i v_i\Vert_i^2 \nonumber \\
\Rightarrow \Vert \nabla_j v_j\Vert_i \geq \Vert \nabla_i v_i\Vert_i
\label{usecs} 
\end{gather}
 At this point we put $t_i=t_j=t$ in \ref{prec}, summing over $i$ and $j$, and using \ref{ker} and \ref{usefu} we get,
\begin{gather}
\displaystyle \sum_{i,j} \int_X \left (\vert \nabla_j v_j \vert_j ^2- \vert\nabla_i v_i \vert_j ^2\right ) dvol \geq 0 \nonumber \\
\Rightarrow \displaystyle \sum_{i<j} \int_X \left (\vert \nabla_j v_j \vert_j^2 + \vert \nabla_i v_i \vert_i ^2- \vert \nabla_i v_i \vert_j ^2 - \vert \nabla_j v_j \vert_i ^2 \right ) dvol \geq 0.
\label{usefour}
\end{gather}
Equation \ref{usefour} in conjunction with \ref{usecs} implied that equality holds in all the inequalities above. Therefore all the $v_i$ are constants and in fact $0$ by normalisation. Note that this may not be true for $t=1$ because equality holding in the inequalities would merely mean that $\nabla_j v_i$ are holomorphic vector fields proportional to each other.
\end{enumerate}
\qed \\

\end{document}